\documentclass{amsart}
\usepackage{amssymb}
\usepackage{color}
\definecolor{darkgreen}{cmyk}{1,0,1,.2}
\definecolor{m}{rgb}{1,0.1,1}
\definecolor{green}{cmyk}{1,0,1,0}

\definecolor{test}{rgb}{1,0,0}
\definecolor{cmyk}{cmyk}{0,1,1,0}

\newtheorem{Equation}{}[section]

\newtheorem{example}[Equation]{Example}
\newtheorem{theorem}[Equation]{Theorem}

\newtheorem{lemma}[Equation]{Lemma}
\newtheorem{corollary}[Equation]{Corollary}

\newtheorem{definition}[Equation]{Definition}

\newtheorem{remark}[Equation]{Remark}
\newtheorem{remarks}[Equation]{Remarks}

\def\ch{\operatorname{ch}}

\def\Dir{\operatorname{{\not{\hspace{-0.05cm}\pa}}}}

\def\I{\operatorname{I}}

\def\Hom{\operatorname{Hom}}

\def\oH{\operatorname{H}}

\def\tr{\operatorname{tr}}

\def\vol{\operatorname{vol}}

\def\C{\mathbb C}

\def\N{\mathbb N}
\def\R{\mathbb R}
\def\S{\mathbb S}
\def\Z{\mathbb Z}
\def\T{\mathbb T}
\def\H{\mathbb H}

\def\cA{{\mathcal A}}

\def\cG{{\mathcal G}}

\def\cG{{\mathcal G}}

\def\cH{{\mathcal H}}

\def\cR{{\mathcal R}}
\def\cS{{\mathcal S}}

\def\cT{{\mathcal T}}
\def\cU{{\mathcal U}}

\def\what{\widehat}

\def\dd{\displaystyle}
\def\pa{\partial}

\def\ep{\epsilon}

\begin{document}



\title[Enlargeability, foliations, and positive scalar curvature \today]
{Enlargeability,  foliations, and positive scalar curvature\\
\today}


\author{Moulay-Tahar Benameur}
\address{Institut Montpellierain Alexander Grothendieck, UMR 5149 du CNRS, Universit\'e de Montpellier}
\email{moulay.benameur@umontpellier.fr}

\author[J.  L.  Heitsch \today]{James L.  Heitsch}
\address{Mathematics, Statistics, and Computer Science, University of Illinois at Chicago} 
\email{heitsch@uic.edu}

\thanks{MSC (2010) 53C12, 57R30, 53C27, 32Q10. }

\begin{abstract} We extend the deep and important results of Lichnerowicz, Connes, and Gromov-Lawson which relate geometry and characteristic numbers to the existence and non-existence of metrics of positive scalar curvature (PSC).  In particular, we show: that a spin foliation with Hausdorff homotopy groupoid of an enlargeable manifold admits no PSC metric; that any metric of PSC on such a foliation is bounded by a multiple of the reciprocal of the foliation K-area of the ambient manifold; and that Connes' vanishing theorem for characteristic numbers of PSC foliations extends to a vanishing theorem for Haefliger cohomology classes.  
\end{abstract} 

\maketitle

\tableofcontents

\section{Introduction}
In this paper, we extend the famous results of  Lichnerowicz, \cite{L}, Connes,  \cite{C86}, and Gromov and Lawson, \cite{GL1, GL2, GL3} on the relationship of geometry and characteristic numbers to the existence and non-existence  of metrics of positive scalar curvature (PSC). Let  $F$  be a spin foliation with Hausdorff homotopy groupoid on a compact manifold $M$.  The condition on the homotopy groupoid allows us to employ the index theory for foliations developed in \cite{H95,HL99}.  Our main result is that if $M$ is enlargeable, then $F$ does not admit a PSC metric.  
We also obtain a bound on any PSC metric on the foliation in terms of the foliation K-area of the ambient manifold, \cite{Gromov}, and we extend Connes' vanishing theorem for characteristic numbers of a PSC foliation  to a vanishing theorem for Haefliger cohomology classes. 

In \cite{S, L}, Schr\"{o}dinger and, independently, Lichnerowicz proved that for any bundle of spinors $\cS$ over a spin manifold $M$, the Atiyah-Singer operator $\Dir$ and the connection Laplacian $\nabla^*\nabla$ on $\cS$ are related by 
$$
\Dir^2 \,\, = \,\, \nabla^*\nabla \,\, + \,\, \frac{1}{4}\kappa,
$$
where $\kappa$ is the scalar curvature of $M$, that is $\kappa = -\sum_{i,j =1}^n \langle R_{e_i,e_j}(e_i),e_j \rangle$, where $e_1,...,e_n$ is any local orthonormal framing of the tangent bundle of $M$ and $R$ is the curvature operator on $M$.  This leads immediately to the following.
\begin{theorem}\label{Lich}{\cite{L}}
If $M$ is a compact spin manifold and $\what{A}(M) \neq 0$, then $M$ does not admit any metric of positive scalar curvature.
\end{theorem}

This theorem and its generalizations have important and deep consequences.  Some of the most far reaching were obtained by Connes and  by Gromov and Lawson.  

In the seminal paper \cite{C86}, Connes  proved that for any transversely oriented foliated manifold, integration over the transverse fundamental class yields a well defined map from the K-theory of the canonically associated $C^*$ algebra (the completion of the smooth functions with compact support on the holonomy groupoid of the foliation, see \cite{C79}) to $\C$.  He derived many important consequences from this, including the following extension of Theorem \ref{Lich}.
\begin{theorem}{\cite{C86}}\label{CTh}
If $M$ is a compact oriented manifold with $\what{A}(M) \neq 0$, then no spin foliation of $M$ has a metric of positive scalar curvature. 
\end{theorem}

Note that  $\what{A}(M)$ need not be an integer, since $M$ is not assumed to be spin. 

\medskip 

Recently Zhang, \cite{Z2016}, has proven the following theorem, which shifts the spin assumption to $M$.  

\begin{theorem}{\cite{Z2016}}\label{ZTh}
If $M$ is a compact oriented spin manifold with $\what{A}(M) \neq 0$, then no foliation of $M$ has a metric of positive scalar curvature.
\end{theorem}
For the results of Lichnerowicz and Connes,  $\what{A}(M)$ is the usual Hirzebruch $\what{A}$ genus which occurs in dimensions $4k$.  The result  of Zhang also includes the Atiyah-Milnor-Singer $\Z_2$ $\what{A}$ invariant, which occurs in dimensions $8k+1$ and $8k+2$. See \cite{LM}, II.7. 

In \cite{GL1,GL2,GL3}, Gromov-Lawson introduced the notion of enlargeable manifolds, which includes all solvmanifolds, all manifolds which admit metrics of non-positive sectional curvature,  all sufficiently large 3-manifolds, as well as many families of $K(\pi,1)$-manifolds.  The category of  enlargeable manifolds is closed under products, connected sums (with anything), and changes of differential structure.

\medskip
Recall the following definitions from \cite{GL3}.

\begin{definition}\label{contract}
A  $C^1$ map $f:M \to M'$ between Riemannian manifolds is $\ep$ contracting if $||f_*(v)|| \leq \ep ||v||$ for all tangent vectors $v$ to $M$.
\end{definition}
 
Denote by $\S^n(1)$ the usual $n$ sphere of radius $1$.
\begin{definition}\label{enlarge}
A compact Riemannian $n$-manifold  is enlargeable if for every  $\ep >0$, there is a orientable Riemannian covering  which admits an $\ep$ contracting map onto $\S^n(1)$ which is constant near infinity and has non-zero degree.
\end{definition}

Gromov-Lawson proved several important non-existence theorems, including the following.

\begin{theorem}[\cite{GL3}]
An enlargeable spin manifold does not admit any metric of positive scalar curvature.
\end{theorem}

In this paper, we extend the Gromov-Lawson result as follows.
\begin{theorem}\label{main}
If $M$ is an enlargeable manifold, then no spin foliation of $M$ with Hausdorff homotopy groupoid  has a metric of positive scalar curvature.  
\end{theorem}

In addition, the techniques used here lead immediately to the following results. 

We obtain a bound on how large the scalar curvature on a spin foliation with Hausdorff homotopy groupoid can be which is a multiple of a natural extension of Gromov's K-area of $M$.   See Theorem \ref{Karea}. 

We extend Connes' vanishing theorem for characteristic numbers associated to a foliation with PSC and holonomy equivariant bundles to the vanishing of {\em cohomology classes} associated to a spin foliation with PSC and Hausdorff homotopy groupoid  and bundles whose pull backs to the groupoid are leafwise almost flat.   See Theorem \ref{Connes1} and Corollary \ref{Connes2}.

\begin{example}  The most classical examples of  enlargeable manifolds are tori.  In \cite{Z2016},  Zhang gave an outline of the proof that the famous result of Schoen-Yau, \cite{SY}, and Gromov-Lawson, \cite{GL1}, that there does not exist a PSC metric on any torus, extends to the case of foliations.  An immediate corollary of our result is the proof of this provided the foliation is spin  and has Hausdorff homotopy groupoid.
\end{example}

\begin{example} \label{exennspin} 
Examples of enlargeable (non-spin) oriented manifolds with $\what{A}(M) = 0$ and spin foliations (so accessable by Theorem \ref{main}, but not by Theorem \ref{CTh} nor Theorem \ref{ZTh}) with Hausdorff homotopy groupoids are given by $M=  \T^k \times ( \T^6 \sharp (\S^2\times \C P^2))$, $k \neq 8\ell +3, 8\ell+4$.   See Section \ref{Finalnotes} for comments, details, and a more sophisticated family of examples.
\end{example}

\begin{remarks} In theorem \ref{main}:\\
\indent
$\bullet$ The condition on the homotopy groupoid can be changed to requiring that there is a covering of $M$ so that the induced foliation has Hausdorff holonomy groupoid. Note that a foliation has Hausdorff homotopy groupoid if and only if any covering foliation also does.  If its holonomy groupoid is Hausdorff, so too is the holonomy groupoid of any covering foliation,  but the converse is false. See \cite{CH97}.  

$\bullet$ 
Note that there are no dimension restrictions.  Note also that the homotopy and holonomy groupoids of any Riemannian foliation are Hausdorff.  In addition, there are many important non-Riemannian foliations which also have this property, e.g.\ those in \cite{LP76}, \cite{H78}, and \cite{KT}. 

$\bullet$
Just as in the Gromov-Lawson results, the spin assumption can be weakened to requiring that the foliation induced on some covering of $M$ has a spin structure.  

$\bullet$ The requirement that $f:M \to \S^n(1)$ be $\ep$ contracting can be loosened to only require that it be $\ep$ contracting on  two forms.  See \cite{LM}, Definition 6.6.

$\bullet$ The condition that $M$ be enlargeable can be replaced by the condition that $M$ has an almost flat K-theory class $[E]$ with  $\ch([E]) =\ch_0([E]) + \ch_n([E])$ and $\dd \int_{M} \ch_n([E]) \neq 0$.  If $\dim M = n$ is odd,  replace $M$ by $M \times \S^1$.

\end{remarks}

Our approach combines the techniques of Gromov-Lawson, \cite{GL1,GL2,GL3} with those of \cite{H95,HL99}.  In particular,  to prove Theorem \ref{main}, we need three things: (1) a way to construct a Hermitian bundle on some  covering of $M$ whose curvature is as small as we like and whose Chern character  is non-trivial only in dimension $n$;  (2) an index theory for elliptic operators defined along the leaves of a foliation which satisfies:  (2a) the index of the operator is the same as the ``graded dimension" of its kernel;  (2b) there is a formula for the index involving characteristic classes.  The first requirement is satisfied by assuming that $M$ is enlargeable.  The second and third are provided by the results in \cite{H95,HL99}, which are quickly reviewed in the next section.

The main point is that the results of \cite{H95,HL99} remain valid on coverings of compact manifolds provided the leafwise spectrum of the relevant leafwise Dirac operator $D$ is nice  and  $F$ has Hausdorff homotopy groupoid. 
 If the foliation has PSC, then the spectrum is very nice, in fact there is a gap about zero.   This implies that the Chern character of the index bundle  must be zero.   But this Chern character is the integral over the leaves of $F$ of the usual expression involving characteristic classes.    Then the assumption that $M$ is enlargeable quickly leads to a contradiction, hence the non-existence of PSC  metrics on $F$.

For an excellent exposition of the circle of ideas so briefly mentioned here see \cite{LM}.

\medskip
\noindent
{\em Acknowledgements.}  It is a pleasure to thank Fernando Alcalde Cuesta, Mikhael Gromov, Gilbert Hector, Steven Hurder, Paul Schweitzer, SJ, and Shing-Tung Yau for helpful information, and the referee for cogent comments which helped improve the presentation.
MB  wishes to thank the french National Research Agency for support via the project ANR-14-CE25-0012-01 (SINGSTAR). 

\section{Preliminaries}\label{results}

In this section we briefly recall the results of  \cite{H95,HL99} we will use.  In particular,  $M$ is a smooth compact oriented manifold of dimension $n$, and $F$ is an oriented foliation of $M$ of dimension $p$  with Hausdorff homotopy groupoid.   The tangent bundle of $F$ will be denoted $TF$, and we assume that $TF$ admits a spin structure.  We may assume that both $M$ and $F$ are even dimensional, for if $F$ is  and $M$ is not, we  replace them by $F$  and  $M \times \S^1$, where the new foliation $F$ is just the old $F$ on each $M \times \{x\}$ for $x \in S^1$.
If  $F$ is not, we replace it by $F \times \S^1$ and $M$ by $M \times \S^1$ or $M \times \S^1 \times \S^1 = M \times \T^2$, depending on whether $M$ is not or is even dimensional.

The homotopy groupoid ${\mathcal G}$ of $F$ consists of equivalence classes of paths $\gamma:[0,1]\to M$ such that the image of $\gamma$ is contained in a leaf of $F$.  Two such paths are equivalent if they are homotopic in their leaf with end points fixed.   If $\cG$ is Hausdorff, then the foliation induced by $F$ on any covering of $M$  also has Hausdorff homotopy groupoid,
\cite{CH97}.   This is essential for the results of \cite{H95,HL99} to hold.   

There are two natural maps $r,s:\cG \to M$, namely $r([\gamma]) = \gamma(1)$ and $s([\gamma]) = \gamma(0)$.  The fibers $s^{-1}(x)$ are the leaves of the foliation $F_s$ of $\cG$.  Note that $r:s^{-1}(x) \to M$ is the simply connected  covering of the leaf of $F$ through $x$. 

The (reduced) Haefliger cohomology of $F$, \cite{Haef}, is given as follows.   Let  ${\cU}$ be a finite good cover of $M$ by foliation charts as defined in \cite{HL90}.  For each $U_i \in {\cU}$, let $T_i\subset U_i$ be a transversal and set $T=\bigcup\,T_i$.  We may assume that the closures of the $T_i$ are disjoint.  Let $\cH$ be the holonomy pseudogroup
induced by $F$ on $T$.   Give $\cA^k_c(T)$, the space of k-forms on $T$ with compact support,  the usual $C^\infty$ topology, and denote the exterior derivative by $d_T:\cA^k_c(T)\to  \cA^{k+1}_c(T)$.   Denote by  $\cA^k_c(M/F)$ quotient of $\cA^k_c(T)$ by the closure of the vector subspace generated by elements of the form $\alpha-h^*\alpha$ where $h\in \cH$ and $\alpha\in\cA^k_c(T)$ has support contained in the range of $h$.  The exterior derivative $d_T$ induces a continuous differential $d_H:\cA^k_c(M/F)\to \cA^{k+1}_c(M/F)$.  Note that $\cA^k_c(M/F)$ and $d_H$ are independent of the choice of cover $\cU$.  The associated cohomology theory is denoted $H^*_c(M/F)$ and is called the Haefliger cohomology of $F$.   This definition will be extended to non-compact coverings of $M$ in Section \ref{ncpf}.

Denote by $\cA^{p+k}(M)$ the space of smooth $p+k$-forms on $M$.  As the bundle $TF$ is oriented, there is a continuous open surjective linear map, called integration over the leaves,
$$
\int_F :\cA^{p+k}(M)\longrightarrow \cA^k_c(M/F)
$$
which commutes with the exterior derivatives $d_{M}$ and $d_{H}$, so it induces the map 
$$ 
\int_F :\oH^{p+k}(M;\R) \to \oH^k_c(M/F).
$$
This map is given by choosing a partition of unity $\{\phi_i\}$ subordinate to the cover $\cU$, and setting 
$$
\int_F \omega \,\, = \,\, \sum_i \int_{U_i} \phi_i \omega,
$$
where $\dd \int_{U_i}$ is integration over the fibers of the projection $U_i \to T_i$.  
Note that each integration $\omega \to \dd \int_{U_i} \phi_i \omega$ is essentially integration over a compact fibration, so $\dd \int_F$ satisfies the Dominated Convergence Theorem, even if $M$ is not compact. 

\medskip
Suppose that $E$ is a Hermitian vector bundle over $M$ with Hermitian connection.  Then any associated  generalized Dirac operator defined along the leaves of $F$ may be lifted, using the projection $r:\cG \to M$,  to a generalized Dirac operator $\Dir_{E}$ along the leaves of the foliation $F_s$ of $\cG$ with coefficients in the pulled back bundle  $r^*(E)$.   In \cite{H95},  a Chern character for $\Dir_{E}$, denoted $\ch(\Dir_{E}^+) \in H^*_c(M/F)$, is constructed using the Bismut superconnection $B$  of \cite{H95} for the foliation $F_s$.   Consider the Haefliger form $\dd \int_F tr_s(K(x,x))$ where $K$ is the Schwartz kernel of $e^{-B^2}$ (on the leaves of $F_s$) and $tr_s$ is the usual supertrace.  Here the variable $x \in M$ is taken to be the element in $\cG$ given by the constant path at $x$.

The main results of \cite{H95} are:   $\dd \int_F tr_s(K(x,x))$ is a closed Haefliger form and its class $\ch(\Dir_{E}^+)$ is independent of the metric on $M$;   if  the metric on $M$ is rescaled by $1/t$ with associated Schwartz kernel  $K_t(x,x)$,  then  $\lim_{t \to 0} tr_s(K_t(x,x))  =   \what{A}(TF)\ch(E)$,  and
$$
\ch(\Dir_{E}^+) \,\, = \,\,  \big[ \int_F \what{A}(TF)\ch(E)\big]  \quad \text{in \, $H^*_c(M/F)$.}
$$
In \cite{H95} it was assumed for simplicity that $M$ was compact.  However, the only condition needed is that $(M,F)$ is of bounded geometry, which is also satisfied by all coverings of $(M,F)$.

Denote by $P_0^{E}$ the graded projection onto the leafwise kernel of  $\Dir_{E}^2$, which also determines a Chern character $\ch (P_0^{E})$ in $H^*_c(M/F)$.  One might hope that, as for compact manifolds, $\ch(\Dir_{E}^+) =  \ch(P_0^{E})$.  However, this is not true in general.  See \cite{BHW}.  The main result of \cite{HL99} is that if the spectrum of  $\Dir_{E}^2$ is sufficiently well behaved near $0$ and $(M,F)$ is of bounded geometry, then $\lim_{t \to \infty} tr_s(K_t(x,x))$ exists, and  $\tr_s(K_t(x,x))$ satisfies the Dominated Convergence Theorem for $\dd \int_F$, as $t \to \infty$.  In particular,
$\dd \ch(P_0^{E}) = \int_F \lim_{t \to \infty} tr_s(K_t(x,x))$, so 
$$
 \int_F \what{A}(TF)\ch(E)   \,\, = \,\, \ch(\Dir_{E}^+)   \,\, = \,\,   \ch(P_0^{E}) \quad \text{in \, $H^*_c(M/F)$}.
$$
A special case is when there is a gap about $0$ in the spectrum of $\Dir_{E}^2$, in which case 
 $\lim_{t \to \infty} tr_s(K_t(x,x))  =   0$, so 
$$
\int_F \what{A}(TF)\ch(E)   \,\, = \,\,  \ch(P_0^{E}) \,\, = \,\, 0 \quad \text{in \, $H^*_c(M/F)$.}
$$

\section{Proof of Theorem \ref{main}: the compactly enlargeable case }\label{cecase}

Recall  definitions \ref{contract} and \ref{enlarge} from the introduction.  A manifold is compactly enlargeable if for each $\ep >0$, the covering space can be chosen to be compact.
For simplicity, we do the compactly enlargeable case first, as the general case requires several technical adjustments for its proof.  

Denote by $\cS$ the spin bundle associated to the spin structure on $TF$.   Then  $\cS_E = \cS \otimes E$ is a 
Dirac bundle when restricted to each leaf $L$ of $F$.  Define the canonical section $\cR^E_F$ of $\Hom(\cS_E,\cS_E)$ by the formula
$$
\cR^E_F(\sigma \otimes \phi) \,\, = \,\, \frac{1}{2} \sum_{j,k=1}^p (e_j \cdot e_k \cdot \sigma) \otimes R^E_{e_j,e_k} (\phi), 
$$
where $e_1,...,e_p$ is an orthonormal local framing of the tangent bundle of $L$, $R^E$ is the curvature transformation of $E$, and the dot is Clifford multipliation.

The following is immediate.  The first two  follow directly from the enlargeability property.  The third is proven in the proof of Corollary 5.6, p.\ 307 of   \cite{LM}, IV.5.

\begin{lemma}\label{lemma1}
Suppose that $F$ is a spin foliation of a  compactly enlargeable Riemannian $n$-dimensional manifold $M$, $n$ even.  Then, for each $\ep > 0$, there is a compact  
orientable Riemannian covering $\what{M} \to M$  and a Hermitian bundle $\what{E} \to \what{M}$ so that,
\begin{itemize}
\item  $\ch(\what{E}) = \dim \what{E} + \ch_{n/2}(\what{E})$;
\item  $\dd \int_{\what{M}} \ch_{n/2}(\what{E}) \neq 0$;
\item $||\mathcal{R}^{\what{E}}_{\what{F}}|| \leq \ep$.
\end{itemize}
where  $\what{F}$ is the foliation induced on $\what{M}$ by $F$,  and $\mathcal{R}^{\what{E}}_{\what{F}}$ is the leafwise operator given above on $\what{M}$.
\end{lemma}

The following lemma allows us to utilize the results of \cite{HL99}.

\begin{lemma}\label{NSinvars}
Suppose that  $F$ is spin foliation of a compactly enlargeable manifold $M$ which admits a metric whose restriction to the leaves of $F$ has PSC.  Let $(\what{M}, \what{F})$ and $\what{E}$ be as in Lemma \ref{lemma1}, and denote the homotopy groupoid of $\what{F}$ by $\what{\cG}$. Then there is a gap about $0$ in the spectrum of the leafwise operator $\Dir_{\what{E}}^2$ on $\what{\cG}$.  
\end{lemma}

\begin{proof}
Suppose not and let $\phi$ be in the image of the spectral projection of $\Dir_{\what{E}}^2$ for the interval $[0, \delta]$, with $L^2$ norm  $||\phi|| = 1$.
We may assume that $\delta$ is as small as we please.    Denote the leafwise scalar curvature  for $\what{F}_s$  by $\kappa$ and the pulled back bundle $r^*(\what{E})$ by $\what{E}_r$.  
Now the leafwise operator $\Dir_{\what{E}}^2$ on $\what{\cG}$ satisfies the pointwise equality, see \cite{LM}, II.8.17,
$$
\Dir_{\what{E}}^2 \,\, = \,\, \nabla^* \nabla \,\, + \,\,  \frac{1}{4}\kappa \,\, + \,\, \mathcal{R}^{\what{E}_r}_{\what{F}_s}.
$$
Choose $\ep$ so small that  $\frac{1}{4}\kappa \,\, - \,\,  ||\mathcal{R}^{\what{E}_r}_{\what{F}_s}|| >0$ on $\what{\cG}$.   Choose $c>0$ so that $c\I \leq \frac{1}{4}\kappa \,\, - \,\,  ||\mathcal{R}^{\what{E}_r}_{\what{F}_s}||$ on $\what{\cG}$. Then on each leaf  $s^{-1}(\what{x})$ of $\what{F}_s$, $\what{x} \in \what{M}$,  we have for all small positive $\delta$, 
$$
\delta \,\, \geq \,\, ||\Dir_{\what{E}}^2(\phi)||    \,\, = \,\,  ||\Dir_{\what{E}}^2(\phi)|| \,   ||\phi||   \,\, \geq \,\,  \langle \Dir_{\what{E}}^2(\phi), \phi \rangle  
\,\, = \,\, 
\int_{s^{-1}(\what{x})} \langle \nabla^*\nabla\phi , \phi \rangle \,\, + \,\, 
\int_{s^{-1}(\what{x})} \langle (\frac{1}{4}\kappa \,\, + \,\, \mathcal{R}^{\what{E}_r}_{\what{F}_s})\phi, \phi \rangle  \,\, = \,\, 
$$
$$
\int_{s^{-1}(\what{x})} ||\nabla\phi||^2 \,\, + \,\, 
\int_{s^{-1}(\what{x})} \langle (\frac{1}{4}\kappa \,\, + \,\, \mathcal{R}^{\what{E}_r}_{\what{F}_s})\phi, \phi \rangle 
\,\, \geq \,\, ( \int_{s^{-1}(\what{x})} ||\nabla\phi||^2) \,\, + \,\, c||\phi||^2 \,\, \geq \,\, c,
$$
an obvious contradiction.
\end{proof}

\begin{remark}
Note that the same proof shows that there is a spectral gap for $\Dir^2$ on $\what{\cG}$.  For this result, we do not need $M$ to be compactly enlargeable.  
\end{remark} 

\medskip
\noindent
{\bf  Proof of Theorem \ref{main}: the compactly enlargeable case.}  The proof is by contradiction.  So assume that  $M$  admits a metric whose restriction to the leaves of $F$ has PSC.  By Lemma \ref{NSinvars}, both $\Dir^2$ and $\Dir^2_{\what{E}}$  have spectral gaps about $0$.  By the results quoted in Section \ref{results}, we have 
$$
\dd  \int_{\what{F}} \what{A}(T\what{F}) \ch(\what{E})   \,\, = \,\, \ch( \Dir_{\what{E}}^+)   \,\, = \,\,  \ch(P^{\what{E}}_0)  \,\, = \,\,  0,
$$ 
and 
$$
\dd  \int_{\what{F}} \what{A}(T\what{F})   \,\, = \,\, \ch(\Dir^+)   \,\, = \,\,  \ch(P_0)  \,\, = \,\,  0,
$$ 
in the Haefliger cohomology of $\what{F}$.   For the closed Haefliger current $C$  which is integration over a complete transversal (Connes'  transverse fundamental class), we have 
$$
0 \,\, = \,\, \langle C, \int_{\what{F}} \what{A}(T\what{F}) \ch(\what{E}) \rangle \,\, = \,\,  (\dim \what{E}) \langle C, \int_{\what{F}} \what{A}(T\what{F})  \rangle \,\, + \,\,  \langle C, \int_{\what{F}}  \ch_{n/2}(\what{E}) \rangle.  
$$
As noted above,  $\dd  \int_{\what{F}} \what{A}(T\what{F}) = 0$, so the first term disappears, but the second term is
$\dd \int_{\what{M}} \ch_{n/2}(\what{E})$, which is non-zero, a contradiction. 

\section{Proof of Theorem \ref{main}: the general case} \label{ncpf}

For the general case, we assume that the Riemannian cover $\pi:\what{M} \to M$ is non-compact, and we adapt the proof of Gromov-Lawson as given in \cite{LM}, Section IV.6.  This involves extending the results of \cite{H95,HL99} to a non-compact covering space of $M$.  In \cite{LM}, they develop a relative index theory in order to deal with non-compact manifolds.  We are able to avoid having to do this by using Haefliger cohomology, which breaks up an integral over a non-compact manifold into a countable number of local integrations over compact fibrations.  Thus arguments which work for compact manifolds extend to non-compact manifolds, provided we have a good definition for Haefliger cohomology on non-compact manifolds.

For the paper \cite{H95}, the main change is in the definition of the Haefliger cohomology for the foliation $\what{F}$ of $\what{M}$.  The (reduced) Haefliger cohomology of $\what{F}$ is given as follows.  The good cover ${\cU}$ of $M$ determines a good cover $\what{\cU}$ of $\what{M}$  which consists of all open sets $\what{U}_{i,j}$ so that 
$\pi:\what{U}_{i,j} \to U_i$ is a diffeomorphism, where $U_i \in \cU$.   That is $\cup_j \what{U}_{i,j} = \pi^{-1}(U_i)$.  
The transversal $T_i  \subset U_i$ determines the transversal $\what{T}_{i,j}  \subset \what{U}_{i,j}$,  and the closures of the $\what{T}_{i,j}$ are disjoint.    The  space $\cA^k_c(\what{T})$ consists of all smooth k-forms on 
$\cup_{i,j} \what{T}_{i,j}$ which have compact support in each $\what{T}_{i,j}$, and such that they are uniformly bounded in the usual $C^\infty$ topology.   As above, we have the exterior derivative $d_{\what{T}}:\cA^k_c(\what{T})\to  \cA^{k+1}_c(\what{T})$, and the integration  $\dd \int_F:\cA^{p+k}_b(\what{M}) \to \cA^k_c(\what{T})$, where 
$\cA^{p+k}_b(\what{M})$ is the space of smooth uniformly bounded, in the usual $C^\infty$ topology,  $p+k$-forms on $\what{M}$.

The holonomy pseudogroup $\what{\cH}$ acts on $\cA^k_c(\what{T})$ just as $\cH$ does on  $\cA^k_c(T)$.  
Denote by  $\cA^k_c(\what{M}/\what{F})$ the quotient of $\cA^k_c(\what{T})$ by the closure of the vector subspace $\what{V}$ generated by elements of the form $\alpha-h^*\alpha$ where $h\in \what{\cH}$ and $\alpha\in\cA^k_c(\what{T})$ has support contained in the range of $h$.  We need to take care as to what ``the vector subspace $\what{V}$ generated by elements of the form $\alpha-h^*\alpha$" means.  This is especially important in the proof of Lemma 3.12 of \cite{H95}.   Members of $\what{V}$ consist of possibly infinite sums of elements of the form 
$\alpha-h^*\alpha$, with the following restriction: for each member of $\what{V}$,  there is $n \in \N$ so that the number of elements of that member having the domain of $h$ contained in any $\what{T}_{i,j}$ is less than $n$.

The exterior derivative $d_{\what{T}}$ induces a continuous differential $d_{\what{H}}:\cA^k_c(\what{M}/\what{F})\to \cA^{k+1}_c(\what{M}/\what{F})$.  Note that $\cA^k_c(\what{M}/\what{F})$ and $d_{\what{H}}$ are independent of the choice of cover $\cU$.  The associated cohomology theory is denoted $H^*_c(\what{M}/\what{F})$ and is called the Haefliger cohomology of $\what{F}$. 

Let $f:\what{M} \to \S^n$ be an $\ep$ contracting map which is constant near infinity and has non-zero degree.
Because $f$ is constant near infinity, the pull back of any bundle on $\S^n$ under $f$ is a trivial bundle off some compact subset $K$ of $\what{M}$.  Using the Chern-Weil construction of characteristic classes, we have that the Chern character of such a bundle has compact support contained in $K$ and it is immediate that Lemma \ref{lemma1} extends to $\what{M}$.  Let $\what{E} \to \what{M}$ be a bundle satisfying Lemma \ref{lemma1}, which is the pull back of a bundle over $\S^n$.  Denote by $\I^{n/2}$  the trivial Hermitian bundle over $\what{M}$ of $\C$ dimension $n/2$.   We may assume that the connection on $\what{E}$ used to construct its characteristic classes is compatible with that on $\I^{n/2}$ off $K$.  Thus the K-theory class $[\what{E}] - [\I^{n/2}]$ is supported on $K$, and its Chern character is $\ch([\what{E}]) - n/2 = \ch_{n/2}(\what{E})$.  

The arguments in \cite{H95,HL99} are arguments on the homotopy groupoid of $\what{F}$.  As noted above, the fact that  $(\what{M}, \what{F}) $ is of bounded geometry implies that those results are equally valid here.  In particular we have that,  in $H^*_c(\what{M}/\what{F})$,
$$
\int_{\what{F}} \what{A}(T\what{F})\ch(\what{E}) \,\, = \,\, \ch(\Dir_{\what{E}}^+) 
\quad \text{and} \quad
\int_{\what{F}} \what{A}(T\what{F})\ch(\I^{n/2}) \,\, = \,\, \ch(\Dir_{\I^{n/2}}^+).
$$ 

The proof of Lemma \ref{NSinvars} works equally well here.  In that lemma we are working on the leaves of $F_s$, which are the same as the leaves of $\what{F}_s$ since we are working on the homotopy groupoids.  In particular they are the universal covers of the leaves of $F$.  Thus  there are spectral gaps about $0$ for   $\Dir_{\what{E}}^2$ and $\Dir_{\I^{n/2}}^2$.  Using this fact and applying the results of \cite{H95,HL99}, extended to foliations and manifolds of bounded geometry, we have $\lim_{t \to \infty} \tr_s(K_t(\what{x},\what{x})) = 0$  for both  $\Dir_{\what{E}}^2$ and $\Dir_{\I^{n/2}}^2$.  In addition $\tr_s(K_t(\what{x},\what{x}))$ satisfies the Dominated Convergence Theorem for $\dd \int_{\what{F}}$ as $t \to \infty$.  Thus 
$$
\int_{\what{F}} \what{A}(T\what{F})\ch(\what{E}) \,\, = \,\, \ch(\Dir_{\what{E}}^+) \,\, = \,\,  \lim_{t \to \infty}  \int_{\what{F}}\tr_s(K_t(\what{x},\what{x}))  \,\, = \,\, \int_{\what{F}} \lim_{t \to \infty} \tr_s(K_t(\what{x},\what{x}))
\,\, = \,\,  0,
$$ 
in $H^*_c(\what{M}/\what{F})$, just as in Section \ref{results}, and similarly for $\Dir_{\I^{n/2}}^2$.   
So 
$$
\int_{\what{F}} \what{A}(T\what{F})\ch(\what{E})   \,\, = \,\,  \int_{\what{F}} \what{A}(T\what{F})\ch(\I^{n/2}) \,\, = \,\,  0, 
$$
in $H^*_c(\what{M}/\what{F})$.   
Now
$$
\int_{\what{F}} \what{A}(T\what{F})\ch_{n/2}(\what{E}) \,\, = \,\, \int_{\what{F}} \what{A}(T\what{F})\ch(\what{E})  -   \int_{\what{F}} \what{A}(T\what{F})\ch(\I^{n/2}) \,\, = \,\, 0.
$$
As the dimension of $\what{M} = n$,  $\what{A}(T\what{F})\ch_{n/2}(\what{E}) = \ch_{n/2}(\what{E})$,  so we have
$\dd \int_{\what{F}} \ch_{n/2}(\what{E}) \,\, = \,\,  0$ in $H^*_c(\what{M}/\what{F})$.   

The Haefliger class $\dd \int_{\what{F}} \ch_{n/2}(\what{E})$ is compactly supported in $H^*_c(\what{M}/\what{F})$, since 
$\ch_{n/2}(\what{E})$ is compactly supported in $H^n(\what{M}, \R)$.  
Just as in the case of a compact covering of $M$, we may pair a compactly supported Haefliger class with the closed Haefliger current $C$  which is integration over the complete transversal $\what{T}$ to obtain
$$
0 \,\, = \,\  \langle C, \int_{\what{F}}  \ch_{n/2}(\what{E}) \rangle  \,\, = \,\  \int_{\what{M}}  \ch_{n/2}(\what{E})  \,\, \neq \,\  0,
$$
a contradiction.

\section{Final notes}\label{Finalnotes}

\noindent {\bf K-area and positive scalar curvature. }  Theorem \ref{main} may be generalized as follows.  Recall the operator $\mathcal{R}^{\what{E}_r}_{\what{F}_s}$ from Section \ref{cecase}.   The following definition is an obvious generalization of K-area due to Gromov, \cite{Gromov}.
\begin{definition}   Suppose $M$ is a compact manifold.  The covering spin foliation K-area $K_{cs}(M,F)$ is the supremum of $(||\mathcal{R}^{\what{E}_r}_{\what{F}_s}||)^{-1}$ over all spin foliations $F$ of $M$ and all bundles $\what{E}$ with compact support on oriented covers $\what{M}$ of $M$  which satisfy $c(\what{E}) = c_0(\what{E}) + c_{n/2}(\what{E})$  and  $\dd\int_{\what{M}}c_{n/2}(E) \neq 0$.
\end{definition}
Note that any enlargeable manifold has $K_{cs}(M,F) = \infty$.  The following is immediate.
\begin{theorem} \label{Karea}  A compact manifold $M$ does not admit any spin foliation F with Hausdorff homotopy groupoid and with scalar curvature $\kappa(F) > 4/K_{cs}(M,F)$ everywhere on $M$. 
\end{theorem}

\noindent {\bf Connes' vanishing results.}  
Connes has vanishing results for characteristic numbers of foliations with PSC.  One of which is 
\begin{theorem}{(Corollary 8.3 \cite{C86})}\label{CTh2}
Suppose that $F$ is a spin foliation of a compact oriented manifold $M$, and that $F$ admits a  PSC metric. Let $\cR$ be the subring of $H^*(M;\R)$ generated by the Chern characters of holonomy equivariant bundles.  Then $\langle  \what{A}(TF)\omega, [M] \rangle = 0$ for all $\omega \in \cR$.
\end{theorem}
Note that any bundle associated to the normal bundle $\nu$ of $F$ is holonomy equivariant.   The important property for our approach is that the pull back of the bundle  under $r:\cG \to M$ be leafwise almost flat along $F_s$.  Holonomy equivariant bundles  have this property.  Using the techniques of this paper, it is straightforward to extend this result to the vanishing of Haefliger cohomology classes as follows.
\begin{theorem}\label{Connes1}
Suppose that $F$ is a spin foliation of a compact oriented manifold $M$ with Hausdorff homotopy groupoid, and that $F$ admits a PSC metric. Let $\cT$ be the subring of $H^*(M;\R)$ generated by the Chern characters of bundles whose pull back  under $r:\cG \to M$ is leafwise almost flat along $F_s$.  Then $\dd \int_F  \what{A}(TF)\omega = 0$ in $H^*_c(M/F)$ for all $\omega \in \cT$.
\end{theorem}
\begin{corollary}\label{Connes2}
Let $F$ and $\cT$ be as above.  Then for any Haefliger current $\beta$, $\dd \langle \beta,  \int_F  \what{A}(TF)\omega\rangle = 0$.
\end{corollary}
Taking $\beta$ to be the Haefliger current given by integration over a complete transversal gives Theorem \ref{CTh2}  for foliations with Hausdorff homotopy groupoid, with $\cR$ replaced by $\cT$.   In particular, taking $\omega = \what{A}(\nu)$ gives  Theorem \ref{CTh} 
for such foliations.

\medskip
\noindent {\bf Details of Example \ref{exennspin}.}   Recall that $M=  \T^k \times ( \T^6 \sharp (\S^2\times \C P^2))$.  Since $\T^k$ and $\T^6$ are enlargeable, so is $M$.    
Since $\what{A}$ is multiplicative and $\what{A}(\T^k) = 0$, $\what{A}(M) = 0$.  The manifold  $\T^6 \sharp (\S^2\times \C P^2)$  is well known to be non-spin  (also its universal cover is non spin).   Now the Stiefel-Whitney classes of these manifolds satisfy 
$$
w_1(\T^k) \,\, = \,\, w_2(\T^k) \,\, = \,\, w_1(T^6 \sharp (\S^2\times \C P^2)) \,\, = \,\, 0  \quad \text{  and }  \quad 
w_2(T^6 \sharp (\S^2\times \C P^2)) \,\, \neq \,\, 0. 
$$
Denote the obvious projections by $\pi_1$ and $\pi_2$.   Then it is clear that $w_1(M) = 0$, and 
$$
w_2(M)  \,\, = \,\, \pi_1^*(w_2(\T^k)) \,\, + \,\, 
 \pi_1^*(w_1(\T^k))   \pi_2^*(w_1(T^6 \sharp (\S^2\times \C P^2)))   \,\, + \,\,  \pi_2^*(w_2(T^6 \sharp (\S^2\times \C P^2))) 
\,\, = \,\,
$$
$$
 \pi_2^*(w_2(T^6 \sharp (\S^2\times \C P^2))) \,\, \neq \,\, 0. 
$$

Next, we need spin foliations of $M$ with Hausdorff homotopy groupoids.  In particular,  there are the ``constant slope" foliations given by Lie sub-algebras of the Lie algebra of $\T^k$.    In increasing generality, we can take  the foliations defined by:  nowhere zero vector fields $X_1,...,X_{\ell}$, with $X_i$ on $\T^{k_i}_i$, with $T^{k_1}_1 \times \cdots \times T^{k_{\ell}}_{\ell} \subset \T^k$;  
nowhere zero, linearly independent, commuting vector fields $X_1,...,X_{\ell}$  on $\T^k$;    if  $Y_1,...,Y_{\ell}$ are arbitrary commuting vector fields on $T^6 \sharp (\S^2\times \C P^2)$ (one can even take $Y_1 = \cdots =Y_{\ell}$), the $X_i$ as immediately above, and  $a_1,...,a_{\ell} \in \R$, then the foliation generated by the vector fields $X_1 +a_1Y_1, ..., X_{\ell} + a_{\ell}Y_{\ell}$ satisfy our conditions.  

For more sophisticated examples we have the following.
\begin{example}  The following examples, which have non-zero secondary characteristic classes which vary linearly independently,  may be substituted for $\T^k$ in Example \ref{exennspin}.  See \cite{H78}, Section 5,  for more details. 

Let $G =  SL_2\R  \times \cdots \times SL_2\R$ (k copies), 
$K =  SO_2  \times \cdots \times SO_2$ (k copies), and $\Gamma = \Gamma_1 \times \cdots \times \Gamma_k$, where $\Gamma_i \subset SL_2\R$ is a discrete subgroup with $\Gamma_i\backslash SL_2\R /SO_2$ a surface of higher genus.  

Let $F_0$ be the foliation of $G \times (\R^{2k}- \{0\})$ whose leaves are given by 
$$
L_v \,\, = \,\, \{(g, g^{-1}v) \, | \, g \in G\},   \quad v \in \R^{2k}- \{0\}.
$$
Denote the coordinates on $\R^{2k}$ by $(x_1,y_1,...,x_k,y_k)$.   Choose non-zero  $\lambda_1, ..., \lambda_k \in \R$, and set 
$$
X_{\lambda} \,\, =  \,\, \sum_{i=1}^k  \lambda_i(x_i \pa/\pa x_i + y_i \pa/\pa y_i).
$$  
Then $X_{\lambda}$ is transverse to $F_0$ and its flow takes leaves to leaves, so $F_0$ and $X_{\lambda}$ define a foliation $F$ on $G \times (\R^{2k}- \{0\})$.  

Let $G$ and $K$ act on $G \times (\R^{2k}- \{0\})$  by 
$$
g_1 \cdot (g,v) \,\, = \,\,   (g_1g,v)   \quad \quad \text{and}  \quad \quad (g,v) \cdot k \,\, = \,\,   (gk,k^{-1}v).
$$
The action of $K$ preserves each leaf of $F$ and the action of $G$ takes leaves of $F$ to leaves of $F$.

Set 
$$
M \,\, = \,\,  \Gamma \backslash  G \times_K (\R^{2k} - \{0\})/\Z,
$$
where $n \in \Z$ acts by the time $n$ flow $\varphi_n$ of the radial vector field $X = \sum_{i=1}^k x_i \pa/\pa x_i + y_i \pa/\pa y_i $ on $\R^{2k}$.  The foliation $F$ induces a foliation on $ \Gamma \backslash  G \times_K (\R^{2k} - \{0\})$ which is preserved by $X$, so also by the action of $\Z$.  Denote the induced foliation on $M$ by $F_{\lambda}$.

Note the following. 
\begin{itemize}
\item  $M \simeq  \Gamma \backslash  G \times_K (\S^1 \times \S^{2k-1})$.   A fundamental domain for $M$ in 
$\Gamma \backslash  G \times_K (\R^{2k} - \{0\})$ is given by 
$$
\{ \varphi_t(\Gamma \backslash  G \times_K  \S^{2k-1})  \, | \, t \in [0,1)\}.
$$  
The diffeomorphism $M \to  \Gamma \backslash  G \times_K (\S^1 \times \S^{2k-1})$ sends $\varphi_t(\Gamma \backslash  G \times_K  \S^{2k-1})$ to $\Gamma \backslash  G \times_K (\{t\} \times  \S^{2k-1})$, i.e.
$M \simeq  (\Gamma \backslash  G \times_K \times \S^{2k-1}) \times \S^1.$
Since $\what{A}$ is multiplicative, $\what{A}(M) = 0$.

\item   The foliation induced on $M$ by $F_0$ has leaves which are covers of $\Gamma \backslash  G / K$, so has tangent bundle equivalent to the bundle $\pi^*(T(\Gamma \backslash  G / K))$ where $\pi:M \to \Gamma \backslash  G / K$ is the projection.  Thus the tangent bundle of $F_{\lambda}$ is equivalent to $T\S^1 \oplus \pi^*(T(\Gamma \backslash  G / K))$,  a trivial bundle, so spin.

\item The diffeomorphism of $G \times (\R^{2k} - \{0\})$ given by $(g,v) \to (g,gv)$ takes the foliation  $F_0$ to the natural flat foliation $F_f$  whose leaves are $G \times \{v\}$.   It induces a diffeomorphism  
$M \simeq G/K \times_{\Gamma} (\R^{2k} - \{0\})/\Z$
and the image of $F_{\lambda}$, also denoted $F_{\lambda}$, is induced by $F_f$ and $X_\lambda$.

\item   The foliation induced by $F_f$ and $X_{\lambda}$ on  $ G /K  \times (\R^{2k} - \{0\})$ obviously has Hausdorff homotopy groupoid, and it covers the foliation $F_{\lambda}$, so $F_{\lambda}$ also has Hausdorff homotopy groupoid.

\item  The universal cover of $G/K \times_{\Gamma} (\R^{2k} - \{0\})/\Z$ is $\H^{2k} \times \R^{2k}$, which obviously admits $\ep$ contracting maps onto $\S^{4k}(1)$ which are constant near infinity and have non-zero degree, so $M$ is enlargeable.  
\end{itemize}
\end{example}

The importance of these foliations is the following, see \cite{H78}.  Any  monomial $\Phi(\sigma_1,...,\sigma_{2k})$ in the elementary symmetric  functions, homogenous of degree $2k$, with at least one odd subscript, determines a secondary characteristic class $\Phi(F_{\lambda}) \in H^{4k-1}(M;\R)$, and  
$$
 \int_{\Gamma \backslash  G \times_K \S^{2k-1}}  \hspace{-1.3cm}\Phi(F_{\lambda}) \,\,\, \,\, = \,\, 
\pi^k \vol(\Gamma \backslash G/K) \frac{\Phi(\lambda_1,\lambda_1,  ..., \lambda_k,\lambda_k)}{(\lambda_1 \cdots \lambda_k)^2}.
$$ 

Note that there are other families of examples in [H78] which give rise to examples of this type.

\end{document}